\documentclass[12pt]{amsart}
\usepackage[utf8]{inputenc}
\usepackage[T1]{fontenc}
\usepackage{enumerate}
\usepackage{lmodern, microtype}
\usepackage{mathtools}
\RequirePackage[l2tabu, orthodox]{nag}
\usepackage{amsmath, amsthm, amssymb}    
\usepackage[british]{babel}
\usepackage[%
 pdfpagelabels=true,
 unicode=true
]{hyperref}
\hypersetup{
 pdfauthor={Eva-Maria Müller},
 pdftitle={An Obstruction to Higher-dimensional Kernel of Dirac Operators },
 pdflang={en} 
}
\newtheorem{satz}{Satz}[section] 
\newtheorem{thm}[satz]{Theorem}
\newtheorem{lem}[satz]{Lemma}
\theoremstyle{remark}

\theoremstyle{definition}
\newtheorem{defi}[satz]{Definition}
\newtheorem{bem}[satz]{Remark}
\newtheorem{prob}[satz]{Problem}
\newtheorem{bsp}[satz]{Example}

\date{}
\title{An Obstruction to Higher-dimensional Kernel of Dirac Operators \ }
\author{Eva-Maria M{ü}ller}
\address{Mathematisches Institut, Universit\"at Freiburg, 79104 Freiburg, Germany}
\email{eva-maria.mueller@math.uni-freiburg.de}

\begin{document}
\begin{abstract}
This paper provides a $K$-theoretic obstruction for higher kernel dimension for Dirac operators. For this we use a fibre-wise Dirac operator that gives rise to a family of Fredholm operators representing a class in topological $K$-theory. Then Chern classes of this $K$-class contain some information about the kernel of the operators. 
\end{abstract}

\maketitle

\section{Introduction}

This paper is motivated by the question how one can prescribe the eigenvalues of a Dirac operator. More precisely fix a closed spin manifold (with fixed spin structure). Then choose a finite discrete subset $\{\lambda_1, \dots, \lambda_k\}\subset \mathbb{R}$ with multiplicities $m_1,\dots ,m_k$. Is there a metric such that the Dirac spectrum in the interval $[\lambda_1,\lambda_k]$ is exactly given by $\{\lambda_1,\dots, \lambda_k\}$ with multiplicities $m_1, \dots , m_k$?\\

For the Laplacian acting on functions instead of the Dirac operator the question is answered: Yves Colin de Verdière proved 1987 in \cite{V87}, that it is possible to find such a metric if the dimension of the manifold is greater than three. For the Dirac operator not much is known. Nigel Hitchin proved 1974 in \cite[Theorem 4.5]{Hi74} that for every spin manifold $M$ with dimension $0$ or $\pm 1 (\text{mod }8)$ there is a metric such that the Dirac operator has non-trivial kernel. Later Christian Bär proved in \cite[Theorem A]{B96} that every spin manifold admits a metric with non-trivial kernel if $dim(M)\equiv 3 \text{ mod }4$. Mattias Dahl showed in \cite{D03} that every three-dimensional manifold has a metric such that the Dirac operator only has simple eigenvalues (in the sense that simple means every eigenvalue has multiplicity two which is a relict of representation theory). Later Dahl showed in \cite{D05} that for non-zero real numbers $-L<l_1<\dots <l_n<L$ and a closed spin manifold $M$ with dimension $>3$ there is a metric $g$ such that 
\begin{itemize}
\item for $dim M\equiv 3$ mod $4$: $spec D^g \cap ((-L,L)\setminus \{0\}) = \{l_1,\dots ,l_k\}$ and every eigenvalue is simple
\item for $dim M\not\equiv 3$ mod $4$: $spec D^g \cap ((-L,L)\setminus \{0\}) = \{\pm l_1,\dots ,\pm l_k\}$ and every eigenvalue is simple.
\end{itemize}
Nikolai Nowaczyk proved in \cite{N16} that on a manifold with dimension $0,6,7$ mod $8$ there exist a metric such that the Dirac operator has at least one eigenvalue of multiplicity at least two.\\

This paper shall provide a $K$-theoretic obstruction for higher kernel dimension: Let $X$ be an odd dimensional closed manifold. Let $B$ be a compact manifold and $((D_X)_b)_{b\in B}$ be a family of fibre-wise Dirac operators on $X$. This family defines a map \[F\colon B\rightarrow \widetilde{Fred}, \hspace{0,5cm} b \mapsto F_b\vcentcolon = \frac{(D_X)_b}{\sqrt{1-((D_X)_b^2}}\]
(compare Section \ref{topologie}) where $\widetilde{Fred}$ is the space of self-adjoint Fredholm operators equipped with the topology that makes this map continuous, compare Definition \ref{weak}. This map defines an element in the $K$-theory $K^1(B)$. The odd Chern classes (see Section \ref{obstruction fredholm}) provide for every $k\in \mathbb{N}$ a map \[c_k^{odd}\colon K^1(B)\rightarrow H^{2k-1}(B).\] Using this characteristic classes it is possible to formulate the main result: 

\begin{thm} \label{obstruction_weak topology}

Let $B$ be a compact space, $F\colon B\rightarrow \widetilde{Fred}$ a continuous map representing the class  $[F]$ in $K^1(B)$ and $k\in \mathbb{N}$ so that there is  $0<\varepsilon <1$ with \[\#\left(spec(F_b)\cap (-\varepsilon, \varepsilon)\right)\leq k \hspace{1cm}\text{ for all } b\in B.\]Then for all $i_0<\dots <i_k\in \mathbb{N}$ the equation \[c^{odd}_{i_0}([F])\smile\dots \smile c_{i_k}^{odd}([F])=0\] holds. The sign $\smile$ stands for the cup product.
\end{thm}
Hence, if there is a family of Dirac operator $((D_X)_b)_{b\in B}$ representing a $K$-class such that the cup product of $k$ odd Chern classes does not vanish, then there is $b_0\in B$ such that the kernel dimension of the operator $(D_X)_{b_0}$ is at least $k$.\\
To detect a $k$-dimensional kernel, we see that the base space $B$ has to have a cohomology of at least grad $k^2$. Therefore the unitary group $U(k)$ is a good candidate for the base space.\\
As first example we consider the case of twisted Dirac operators over the circle $S^1$.
\begin{lem}\label{twisted Dirac circle}
Let $k>0$ be a natural number. Let $\Sigma(S^1)$ be the spinor bundle associated to a fixed spin structure on $S^1$. Then there exist a vector bundle $E\rightarrow S^1$ so that the kernel of the twisted Dirac operator $D$ on $E\otimes \Sigma(S^1)$ has dimension at least $k$.  
\end{lem} This is a very easy case, because the twist generates \em{non-geometric} \em degrees of freedom. In future work we plan to find more geometric examples. \\

In Section \ref{obstruction fredholm} we first prove the $K$-theoretic obstruction for general families of Fredholm operators. In Section \ref{continuous fields} and \ref{topologie} we explain how Dirac operators fit in this setting. The main difficulty is to choose a topology on the space of Fredholm operators which is strong enough to classify $K$-theory and such that the family of Dirac operators define a continuous map. Finally we prove Theorem \ref{obstruction_weak topology} and Lemma \ref{twisted Dirac circle}.

\section{$K$-Theoretic Obstruction Theorem for Fredholm Operators}\label{obstruction fredholm}
For an arbitrary infinite dimensional, separable Hilbert space $H$ we define \begin{equation}\label{Defi_Fred}\begin{split}
Fred^0&\vcentcolon =\{F\colon H\rightarrow H| F\text{ Fredholm, } FF^\ast-1\text{ and } F^\ast F-1 \text{ are compact}\}\\
Fred^1&\vcentcolon=\{F\in Fred^0 | F \text{ selfadjoint}\}
\end{split}.\end{equation} We give both space the topology induced by the operator norm. $K$-theory is a generalized cohomology theory with \begin{equation*}\label{i}K^{-n}(B)=\begin{cases}
 [S^n\wedge B,Fred^0]&\text{ if } B \text{ is based }\\
 [S^n \wedge B_+, Fred^0]& \text{ else }
 \end{cases}.\end{equation*}
For an element in $K^0(B)$, the base point of $S^n\wedge B$ (resp. $S^n\wedge B_+$) has to be mapped to an invertible operator.  The $K$-theory satisfies Bott periodicity $K^i(B)\cong K^{i+2}(B)$. In \cite{AS69} Atiyah and Singer showed for a non-based space\[K^1(B)\cong [B_+,Fred^1].\]One important property is that an element $A\in [B_+,Fred^1]$ represents the zero in $K^1(B)$ if the kernel dimension is locally constant (\cite[Theorem 4.1]{E}). In particular $A$ represents zero if $A_b$ is invertible for every $b\in B$. For more information about $K$-theory see for example \cite{A67} or \cite{LM89}.\\
Chern classes are characteristic classes which map each element in $K^0(B)$ to a certain classe in $H^{2i}(B)$ (see for example \cite{LM89}): \[c_i\colon K^0(B)\rightarrow H^{2i}(B).\]  

We only consider the case when $B$ is non-based: Using the suspension isomorphism $H^i(S^1\wedge B_+)\cong H^{i-1}(B_+)$ we define the odd Chern classes: \[c^{odd}_i\colon K^1(B)=K^0(S^1\wedge B_+)\overset{c_i}\rightarrow H^{2i}(S^1\wedge B_+)\overset{\cong}\rightarrow H^{2i-1}(B_+).\] The odd Chern classes inherit all properties of the even Chern classes. 

\begin{lem}\label{computation odd chern}
Let $B$ be a topological space, $W\rightarrow S^1\wedge B_+$ be a vector bundle. The the relation \[ch^{odd}(W)=\sum_{n=1}^\infty \frac{(-1)^{n+1}}{(n-1)!}c_n^{odd}(W)\] holds, where $ch^{odd}$ is the odd Chern character (for the Definition of odd/even Chern character see \cite[Chapter 10.1]{AH61}).
\end{lem}

\begin{proof}
The odd Chern character is the composition of the even Chern character with the suspension isomorphism $S\colon H^i(S^1\wedge B_+)\rightarrow H^{i-1}(B_+)$: \[ch^{odd}\colon K^1(B)=K^0(S^1\wedge B_+)\xrightarrow{ch^{even}}H^{2\ast}(S^1\wedge B_+)\overset{S}\cong H^{2\ast-1}(B_+).\]
Let $c_i(W)\in H^{2i}(S^1\wedge B_+)$, $c_j(W)\in H^{2j}(S^1\wedge B_+)$ be two Chern classes. Since the cup product over a suspension vanishes, we see that $c_i(W)\smile c_j(W)=0$.
We compute \begin{equation}\begin{split}ch^{odd}(W)&=S\circ ch^{even}(W)\\
&=S\left( rk(W)+c_1(W)+\frac{1}{2}(c_1(W)\smile c_1(W)-2c_2(W))+\dots\right)\\
&=S\left( rk(W)+\sum_{n=1}^\infty \frac{(-1)^{n+1}}{(n-1)!} c_n(W)\right)=\sum_{n=1}^\infty \frac{(-1)^{n+1}}{(n-1)!}c^{odd}_i(W)
\end{split}\end{equation}
\end{proof}

To prove the obstruction theorem for families of Fredholm operators we need:

\begin{lem}\label{homotopie_stark}
Let $f\colon\mathbb{R}\rightarrow \mathbb{R}$ be a continuous function such that there is a constant $0<c\leq 1$ with $f|_{\mathbb{R}\setminus (-c,c)}=id$. Let $B$ be a topological space, $A\colon B\rightarrow Fred^1$ be a continuous map. Then the map $f(A)\colon B\rightarrow Fred^1$, given by $f(A)_b=f(A_b)$, is a well defined continuous map. Moreover if there is a homotopy $h$ between $f$ and identity with $h|_{(-c,c)\times [0,1]}=id$, then $f(A)$ and $A$ are homotopic.
\end{lem}
\begin{proof}
Because $f|_{\mathbb{R}\setminus (-c,c)}=id$ and $A_b^2-1$ is compact, the operator $(f(A_b))^2-1$ is compact.\\
Since $b\mapsto A_b$ is continuous with respect to the norm topology, the map $b\mapsto f(A_b)$ is continuous.\\
For the last part: Let $h\colon\mathbb{R}\times [0,1]\rightarrow \mathbb{R}$ be the homotopy. We consider the map $H\colon B\times[0,1]\rightarrow Fred^1$ given by $H(b,t)=h(A_b,t)$. This is a homotopy between $f(A)$ and $A$. 
\end{proof}

\begin{thm} \label{obstruction_strong topology}
Let $B$ be a compact space, $A\colon B\rightarrow Fred^1$ a continuous map representing the class  $[A]$ in $K^1(B)$ and $k\in \mathbb{N}$ so that there is  $0<\varepsilon <1 $ with \[\#(spec(A_b)\cap (-\varepsilon, \varepsilon))\leq k \hspace{1cm}\text{ for all } b\in B.\]Then for all $i_0<\dots <i_k\in \mathbb{N}$ the equation \[c^{odd}_{i_0}([A])\smile\dots \smile c_{i_k}^{odd}([A])=0\] holds. The sign $\smile$ stands for the cup product.
\end{thm}

\begin{proof}
Consider for $j=0,\dots, k$ the sets \[U_j\vcentcolon=\{b\in B |A_b-\frac{j\varepsilon}{k+1} \text{ is invertible}\}.\] Using the Neumann series it is easy to see that the set of invertible operators is open in $Fred^1$. Therefore $U_j\subset B$ is open.
Suppose there is $b\in B$ such that $b\not \in U_j$ for all $j$. Then $\frac{j\varepsilon}{k+1}$ is an eigenvalue of $A_b$ for every $j\in 0,\dots , k$. But 
\[\left| \frac{j\varepsilon}{k+1}\right| \leq\frac{k\varepsilon}{k+1} < \varepsilon.\]
This implies $\#(spec(A_b)\cap (-\varepsilon, \varepsilon))>k$, a contradiction. Therefore $\{U_j\}$ is an open cover of $B$. Let $a_j\vcentcolon=\frac{j\varepsilon}{k+1}\in [0,1)$. We define the continuous map \[g_j\colon\mathbb{R}\rightarrow \mathbb{R},\hspace{1cm} 
x\mapsto \begin{cases}
x 	& \text{if } \vert x \vert\geq 1\\
\frac{x-a_j}{1+a_j}& \text{if } x\leq a_j, \ \vert x \vert \leq 1\\
\frac{x-a_j}{1-a_j}& \text{if } x\geq a_j, \ \vert x \vert \leq 1
\end{cases}.\]
The map $g_j$ satisfies the requirements of Lemma \ref{homotopie_stark}. A linear homotopy gives a homotopy between $g_j$ and the identity. By Lemma \ref{homotopie_stark} there is a well-defined function $A_j\vcentcolon=g_j(A)\colon B\rightarrow Fred^1$ and $A$ is homotopic to $A_j$. Therefore \begin{equation}\label{77}
c^{odd}_{i_j}([A])=c^{odd}_{i_j}([{A}_j]).\end{equation} Moreover $ker((A_j)_b)=ker(A_b-a_j)$ for every $b\in B$. Therefore ${A}_j|_{U_j}$ is invertible, hence  $[{A}_j|_{U_j}]=0\in K^1(U_j)$. This implies \begin{equation}\label{78}c_{i_j}^{odd}([{A}_j])\in im\left( H^{2i_j-1}(B,U_j)\rightarrow H^{2i_j-1}(B)\right).\end{equation}
Using \eqref{77} and \eqref{78} we get \[\begin{split}c^{odd}_{i_0}([A])\smile \dots \smile c^{odd}_{i_k}([A])=&c^{odd}_{i_0}([{A}_0])\smile \dots \smile c^{odd}_{i_k}([{A}_k])\\
&\in im\left(H^\bullet (B,\bigcup U_j)\rightarrow H^\bullet(B)\right).\end{split}\] But $H^\bullet(B,\bigcup U_j)=H^\bullet(B,B)=0$ which finishes the proof.
\end{proof}

Last theorem gives a topological obstruction for higher kernel dimension of Fredholm operators. In this paper we want to use the obstruction to get a lower bound for the kernel dimension of Dirac operators. For this purpose we consider a family of Dirac operators $((D_X)_b)_{b\in B}$ indexed by some compact manifold $B$ coming from a fibre-wise Dirac. We associate to this family the bounded Fredholm operators $(F_b)=\left(\frac{(D_X)_b}{\sqrt{1-(D_X)_b^2}}\right)$. This family represents an element in $K^1(B)$ if there is a continuity condition on $b\mapsto F_b$. We will see in the next sections, that the map $b\mapsto F_b$ is not continuous if we use the norm topology on $Fred^1(B)$. To solve we will describe a weaker topology.

\section{Obstruction Theorem for Dirac Operators}
As mentioned above we want to use the obstruction theorem for high kernel dimension to estimate the kernel dimension of Dirac operator. For this purpose we have to understand, why fibre-wise Dirac operators represent elements in $K$-theory. 
\subsection{Fibre-wise Dirac Operator}\label{fibre-wise Dirac operators}
We first summarize some properties of the Dirac operator. 
Let $X$ be a closed spin-manifold with volume measure, $V\rightarrow X$ be a Clifford bundle with complex inner product. The space $\Gamma(X,V)$ is the pre-Hilbert space (a Hilbert space without completeness condition) of smooth sections with $L^2$ inner product.
The completion of this pre-Hilbert space is the space $L^2(X,V)$.
The Dirac operator is a self-adjoint unbounded operator \[D\colon L^2(X,V)\rightarrow L^2(X,V).\]
Furthermore the Dirac operator admits a spectral decomposition. Since $D$ is self-adjoint, all eigenvalues are real. Therefore the operators $D\pm i$ are invertible operators with bounded inverse $(D\pm i)^{-1}\colon L^2(X,V)\rightarrow dom(D)\hookrightarrow L^2(X,V)$. By the G\r{a}rding inequality and the Rellich theorem the bounded operator $(D^2+1)^{-1}$ is compact. Therefore the operator $\frac{D}{\sqrt{1+D^2}}$ is an element in $Fred^1$. \\

We come to the fibre-wise situation. We only give a short overview. For more details see \cite[Chapter 1]{Wi} or \cite{BGV92}.\\
We consider an oriented Riemannian fibre bundle $X\hookrightarrow M \xrightarrow{\pi} B$. The fibres $\pi^{-1}(b)\cong X$ are closed $n$-dimensional compact spin-manifolds where $n$ is odd. We denote the vertical tangent bundle by $TX\vcentcolon=ker d\pi$. \\

Let $(V,g,\bigtriangledown,c_X)\rightarrow M$ be a vertical Dirac bundle. We associate the fibre-wise Dirac operator \[D_X\vcentcolon=\sum_{i=1}^n c_X (e_i)\bigtriangledown^V_{e_i}\colon\Gamma(M,V)\rightarrow \Gamma(M,V),\]
where $e_1,\dots e_n$ is a local orthonormal frame of $TX$ and $c_X$ is a Clifford multiplication. 

The Dirac operator $D_X$ restricts to fibre-wise sections $\Gamma(M_b, V_b)$. We get a family $b\mapsto (D_X)_b$ where $(D_X)_b$ acts on $\Gamma(M_b, V_b)$. We get a fiberwise inner product (compare \cite[Section 1.3]{E2}) \begin{equation}\label{fiberwise inner product}\langle s,t\rangle_b\vcentcolon=\int_{M_b}\langle s(b),t(b)\rangle dvol_{M_b}(b).\end{equation}
As explained above the fibre-wise Dirac $(D_X)_b$ is then an unbounded, densely defined operator $dom((D_X)_b)\subset L^2(M_b,V_b)\rightarrow L^2(M_b,V_b)$.

We associate to this family of unbounded operators $(D_X)_b$ a family of bounded operators: Via functional calculus we associate a family of bounded operators $F_b\vcentcolon=\varphi((D_X)_b)$, where \[\varphi\colon\mathbb{R}\rightarrow \mathbb{R}, \hspace{1cm} x\mapsto \frac{x}{\sqrt{1+x^2}}.\] 

We consider the operator $F_b$ as a bounded operator $L^2(M_b,V_b)\rightarrow L^2(M_b,V_b)$. This operator is a self-adjoint Fredholm operator, so that $F^2_b-1$ is compact. 

\begin{prob}\label{problem}We see now the first problems: The Hilbert spaces on which the operators $F_b$ act differ by changing $b$. To get a family of Fredholm operators in the sense of Section \ref{obstruction fredholm} we have to identify all Hilbert spaces and find a topology on $Fred^1$ which is strong enough to classify $K$-theory and weak enough such that the map $b\rightarrow F_b$ is continuous. 
\end{prob}
The readers who know already in which sense the fibre-wise Dirac operator represent an element in $K$-theory can skip the next section and jump directly to Section \ref{proof}.
\subsection{Continuous Fields of Hilbert spaces}\label{continuous fields}
To solve Problem \ref{problem}, we follow Dixmier and Douady. They introduced the term \emph{continuous field of Hilbert spaces}, which generalises the vector bundle theory to the infinite dimensional setting. In this paper we discuss only those things of importance for later sections and omit some proofs. For details see \cite{DD}.\\
Let $B$ be a topological space and $(H(b))_{b\in B}$ be a family of (pre-)Hilbert spaces. Let $\langle \cdot,\cdot \rangle_b$ be the scalar product on $H(b)$. We consider the space $\prod_{b\in B} H(b)$. We take an element $x\in \prod_{b\in B} H(b)$ as a map $x\colon B\rightarrow \bigcup_{b\in B} H(b)$ such that $x(b)\in H(b)$. In analogy to bundles we call an element of $\prod_{b\in B} H(b)$ \em a section of $(H(b))_{b\in B}$\em.
\begin{defi}\label{defi_field}
Let $B$ be a topological space. A continuous field of Hilbert spaces over $B$ is a pair $\mathcal{H}=((H(b))_{b\in B},\Gamma)$ where $(H(b))_{b\in B}$ is a family of Hilbert spaces and $\Gamma\subset \prod_{b\in B} H(b)$ so that: \begin{itemize}
\item[(i)] $\Gamma$ is a $C(B,\mathbb{C})$- module, where $C(B,\mathbb{C})$ are the continuous functions,
\item[(ii)] for every $b\in B$ and every $v\in H(b)$, there exist $x\in \Gamma$ so that $x(b)= v$,
\item[(iii)] for every $x\in \Gamma$, the function $B\rightarrow \mathbb{C}$ given by $b\mapsto \langle x(b),y(b)\rangle_b$ is continuous,
\item[(iv)] let $x\in \prod_{b\in B} H(b)$. If for all $b\in B$ and for all $\varepsilon >0$, there is $y\in \Gamma$ and a neighbourhood $U$ of $b$ so that $sup_{a\in U}\Vert x(a)-y(a)\Vert_a < \varepsilon$, then $x\in \Gamma$. 
\end{itemize}
We call the elements of $\Gamma$ the \em continuous sections \em of $\mathcal{H}$.
\end{defi}

\begin{bsp}\label{trivial bundle}
Let $H$ be a Hilbert space, $B$ a topological space. We consider the trivial field of Hilbert spaces with $H(b)=H$ for every $b$. Let $\tilde{\Gamma}$ be the set of continuous functions $B\rightarrow H$. We define $\Gamma\vcentcolon=\{id_B\times x \colon B\rightarrow B\times H \ |\ x\in \tilde{\Gamma}\}$. Then the pair $\mathcal{H}_B\vcentcolon=((H(b))_{b\in B},\Gamma)$ is a continuous field of Hilbert spaces.
\end{bsp}
\begin{bem}\label{remark topology}
Let $\mathcal{H}=((H(b))_{b\in B},\Gamma)$ be a continuous field of Hilbert spaces. Define $Tot(\mathcal{H})$ be the union of the spaces $H(b)$ and $\pi\colon Tot(\mathcal{H})\rightarrow B$ be the canonical projection. \\
It is possible (see \cite[Chapter 2]{DD}) to equip $Tot(\mathcal{H})$ with a topology such that the map $\pi$ is continuous and the elements of $\Gamma$ are exactly the continuous maps $B\rightarrow Tot(\mathcal{H})$.\\ 
If we consider Example \ref{trivial bundle}, it turns out, that the topology constructed in \cite{DD} is exactly the product topology. 
\end{bem}
\begin{defi}\label{prefield}
Let $B$ be a topological space. For every $b\in B$, let $E(b)$ be a pre-Hilbert space and $\overline{E(b)}$ be the closure of $E(b)$. We call the pair $\mathcal{E}\vcentcolon=((E(b))_{b\in B},\Lambda)$, where $\Lambda \subset \prod_{b\in B} E(b)$, a continuous field of pre-Hilbert spaces if \begin{itemize}
\item For every $b\in B$, the set $\Lambda(b)$ is a dense subset of $\overline{E(b)}$. 
\item For every $x\in \Lambda$ the function $B\rightarrow \mathbb{C}$ given by $b\mapsto \langle x(b),y(b)\rangle_b$ is continuous. 
\end{itemize}
\end{defi}

\begin{lem}\cite[Lemma 2.7]{E2}\label{completion fields}
Let $\mathcal{E}=((E(b))_{b\in B},\Lambda)$ be a continuous field of pre-Hilbert spaces, $H(b)\vcentcolon=\overline{E(b)}$ be the closure. Then there is a unique subset $\Gamma\subset \prod H(b)$, so that $\Lambda\subset \Gamma$ and $((H(b))_{b\in B},\Gamma)$ is a continuous field of Hilbert spaces.
\end{lem}
We come back to the situation we are interested in -- the fibre-wise Dirac operator (compare \cite[Example 2.12]{E2}).
\begin{bsp} \label{Fiberwise Dirac continuous field}
Let $X\hookrightarrow M\xrightarrow{\pi} B$ be a fibre bundle with $\pi^{-1}(b)\cong X$ an odd dimensional, closed spin-manifold, $B$ a compact manifold. Let $V\rightarrow M$ be a vertical Clifford bundle. We consider the pre-Hilbert spaces $E(b)\vcentcolon=\Gamma(V_b,M_b)$. Let $\Lambda\subset \prod_{b\in B} H(b)$ be the set $\Gamma(M,V)$. For every $b\in B$, we associate to $s\in \Lambda$ an element $s(b)\vcentcolon=s|_{M_b}\in \Gamma(M_b,V_b)$. By the dominated convergence theorem the map $b\rightarrow \langle s(b),t(b)\rangle_b$ is a continuous map, where $s,t\in \Lambda$. Therefore the pair $\Gamma_B(M,V)\vcentcolon=((E(b))_{b\in B}, \Lambda)$ is a continuous field of pre-Hilbert spaces. The completion (in the sense of Lemma \ref{completion fields}) $L_B^2(M,V)=((L^2(M_b,V_b))_{b\in B},\Gamma)$ is a continuous field of Hilbert spaces. 
\end{bsp}

\begin{defi}(Direct sum)\label{Direct sum}
For every element $i$ of a index set $I$, let $\mathcal{H}_i=((H_i(b))_{b\in B},\Gamma_i)$ be a continuous field of Hilbert spaces over a topological space $B$. Let $H(b)\vcentcolon=\bigoplus_I H_i(b)$ and define \[\Gamma'\vcentcolon=\{\sum_{i\in J}x_i|J\subset I \text{ is finite }, x_i\in \Gamma_i\}.\] Let $\Gamma$ be the closure of $\Gamma'$ with respect to the property (iv) of Definition \ref{defi_field}. Then we write $\bigoplus_I \mathcal{H}_i$ for the continuous field $((H(b))_{b\in B},\Gamma)$. 
\end{defi}
\begin{thm}\cite[Theor\`{e}me 4]{DD}\label{trivial theorem}
Let $\mathcal{H}=((H(b))_{b\in B},\Gamma)$ be a continuous field of Hilbert spaces with separable fibres over a paracompact space $B$. Let $H$ be a Hilbert space with $dim(H)=\infty$. And let $\mathcal{H}_B=B\times H$ be the trivial Hilbert field (see Example \ref{trivial bundle}). Then $\mathcal{H}\oplus \mathcal{H}_B\cong \mathcal{H}_B$.
\end{thm}
 
\subsubsection{Morphisms of Continuous Fields}
Now we are able to define a morphism of continuous fields of Hilbert spaces.
 
\begin{defi}\label{homomorphism}
Let $B$ be a topological space, $\mathcal{H}=((H(b))_{b\in B},\Gamma)$ be a continuous field of Hilbert space (resp. pre-Hilbert space) and let $\mathcal{H'}=((H'(b))_{b\in B},\Gamma')$ be a continuous field of Hilbert spaces. Let $\varphi=(\varphi_b)_{b\in B}$ be a family of linear operators $\varphi_b\colon H(b)\rightarrow H'(b)$. The map $\varphi$ is a homomorphism of continuous Hilbert (resp. pre-Hilbert) fields if 
\begin{itemize} 
\item for every $x\in \Gamma\Rightarrow\varphi\circ x\in \Gamma'$ and
\item the map $B\rightarrow \mathbb{C}$, given by $b\mapsto \Vert \varphi_b\Vert_b$ is locally bounded.
\end{itemize}
\end{defi}

Dixmier and Douady defined the term homomorphism differently. Both definitions are equivalent by \cite[Proposition 5]{DD}. 
\begin{lem}\cite[Lemma 2.7]{E2}\label{extenison}
Let $\mathcal{E}=((E(b))_{b\in B}, \Lambda)$ be a continuous field of pre-Hilbert space, $\mathcal{H}=((H(b))_{b\in B},\Gamma)$ a continuous field of Hilbert spaces and $\varphi\colon  \mathcal{E}\rightarrow \mathcal{H}$  a continuous field of pre-Hilbert spaces. 
Then there is a unique homomorphism Hilbert spaces $\overline{\varphi}\colon  \overline{\mathcal{E}} \rightarrow \mathcal{H}$ which coincides on $E(b)$ with $\varphi_b$. Here $ \overline{\mathcal{E}}$ stands for the extension of $\mathcal{E}$ constructed in Lemma \ref{completion fields}.
\end{lem}
\begin{bsp}\label{trivial bundle homo}
Let $\mathcal{H}_B$ be the trivial bundle of Example \ref{trivial bundle}. A fibre-wise linear map $\varphi\colon  \mathcal{H}_B\rightarrow \mathcal{H}_B$ is uniquely defined by a map $F\colon B\rightarrow \mathcal{B}(H)$ where $\mathcal{B}(H)$ is the set of linear, continuous operators on $H$. We give $\mathcal{B}(H)$ the weakest topology so that the evaluation $ev\colon  \mathcal{B}(H)\times H\rightarrow H$ is continuous. Then $\varphi$ is a homomorphism if and only if $F$ is continuous.
\end{bsp}
As a second example we consider the fibre-wise Dirac operator. 
\begin{bsp}\label{fiberwise Dirac morphism}
We are in the situation of Example \ref{Fiberwise Dirac continuous field}. We consider now the family of fibre-wise Dirac operators $((D_X)_b)_{b\in B}$ explained in Section \ref{fibre-wise Dirac operators}. Via functional calculus we associate to $(D_X)_b$ a bounded operator \[F_b\vcentcolon=\frac{(D_X)_b}{\sqrt{1+((D_X)_b)^2}}\colon L^2(V_b,M_b)\rightarrow L^2(V_b,M_b)\](compare Section \ref{fibre-wise Dirac operators}). The function $b\mapsto \Vert F_b\Vert_b$ is globally bounded by $1$. \\
Let $s\in \Lambda$ be a smooth section of the bundle $V\rightarrow M$. Then $\frac{D_X}{\sqrt{1+D_X^2}}(s)$ is again a smooth section. Hence, $\frac{D_X}{\sqrt{1+D_X^2}}(s)$ is a continuous section of the continuous field $L^2_B(M,V)$. By Lemma \ref{extenison} the family $\{F_b\} $ defines a homomorphism of the continuous fields $L^2_B(M,V)\rightarrow L^2_B(M,V)$.
\end{bsp}
\begin{defi}
We call an endomorphism $(\varphi_b)_{b\in B}$ of continuous Hilbert fields invertible (resp. adjointable) if the linear maps $\varphi_b$ are invertible (resp. adjointalbe) for every $b\in B$ and the collection $(\varphi_b^{-1})$ (resp. $(\varphi_b^\ast))$ is a morphism.\\
An invertible endomorphism between continuous fields of Hilbert spaces is called isomorphism. 
\end{defi}
\begin{bsp}\label{fiberwise Dirac selfadjoint}
By Example \ref{fiberwise Dirac morphism} the fibre-wise Dirac operator $D_X$ gives rise to a morphism of continuous Hilbert spaces $(F_b)_{b\in B}$. Since every operator $(D_X)_b$ is self-adjoint, the bounded operators $F_b$ are self-adjoint as well. Hence, the operators $(F_b)_{b\in B}$ form a self-adjoint morphism of continuous Hilbert fields. 
\end{bsp}
We see that the theory of continuous fields of Hilbert spaces is a good theory for the fibre-wise situation.\\

Ebert extended the functional calculus theorem for self-adjoint operators to the field case.  
\begin{thm}(Functional calculus)\cite[Theorem 2.30]{E2}\label{functional calculus}
Let $\mathcal{H}=((H(b))_{b\in B},\Gamma)$ be a continuous field of Hilbert spaces, $\varphi=(\varphi_b)\colon \mathcal{H}\rightarrow \mathcal{H}$ be a self-adjoint morphism and $f\colon \mathbb{R}\rightarrow \mathbb{R}$ be a continuous bounded map. Then the collection $(f(\varphi_b))_{b\in B}$ is a morphism of continuous Hilbert spaces.
\end{thm}

\subsubsection{Compact and Fredholm homomorphism of Continuous fields}
The notion \em compact operator \em is more difficult to formulate. We follow here Dixmier and Douady \cite[Chapter 22]{DD} and Ebert \cite{E1}.\\

Let $\mathcal{H}=((H(b))_{b\in B},\Gamma)$ be a continuous field of Hilbert spaces. Let $\mathcal{K}(H(b))$ be the Banach space of compact operators on $H(b)$. For $x,y\in \Gamma$ we define the operators $\theta_{x,y}$ of rank $\leq 1$: \[\theta_{x,y}\colon  \mathcal{H}\rightarrow \mathcal{H}, \hspace{1cm} (b,v)\mapsto \left(b,\langle x(b),v\rangle_b\ y(b)\right)\] where $v\in H(b)$. Define \[\Lambda\vcentcolon=span\{\theta_{x,y}|x,y\in \Gamma\}.\]
\begin{defi}\label{compact defi}
Let $\mathcal{H}=((H(b))_{b\in B},\Gamma)$ be a continuous field of Hilbert spaces. The homomorphism $\varphi\colon  \mathcal{H}\rightarrow \mathcal{H}$ is called compact if for every $b\in B$ and $\varepsilon >0$ there is a neighbourhood $U\subset B$ of $b$ and $G\in \Lambda$ such that $\Vert \varphi_y-G_y\Vert_y \leq \varepsilon$ holds for every $y\in U$.
\end{defi}
\begin{lem}\cite[Remark 2.16]{E1}\label{compact lem}
Let $\mathcal{H}=((H(b)),\Gamma)$ be a continuous field of Hilbert spaces, let $\varphi, \psi\colon  \mathcal{H}\rightarrow \mathcal{H}$ be two homomorphisms such that $\varphi$ is compact and $\psi$ is adjointable.
Then $\psi\circ \varphi$ and $\varphi\circ \psi$ are compact. 

\end{lem}

\begin{bsp}\label{trivial compact}
Let $B$ be a topological space, $H$ a Hilbert space. We consider the trivial continuous field of Hilbert spaces $\mathcal{H}_B$ of Example \ref{trivial bundle}. Let $\varphi\colon  \mathcal{H}_B\rightarrow\mathcal{H}_B$ be a morphism. This morphism is uniquely defined by a map $F\colon B\rightarrow \mathcal{B}(H)$ (compare Example \ref{trivial bundle homo}). We claim that $\varphi$ is compact if and only if $F\colon  B\rightarrow \mathcal{K}(H)$ is a continuous map (The topology on the right hand side is given by the operator norm topology.):\\
We first consider the case if $F$ is continuous. Let $\{v_i\}$ be a complete orthonormal system of $H$. Define $s_i,t_i\in \Gamma$ by $s_i(b)\vcentcolon=(b,v_i)$ and $t_i(b)\vcentcolon=(b,F(b)(v_i))$ ($s_i,t_i$ are elements of $\Gamma$ by Remark \ref{remark topology}). We define \[\theta_k\vcentcolon=\sum_{i=1}^k \theta_{s_i,t_i}\in \Lambda.\] Then $\theta_k$ approximates $\varphi$ globally, hence $\varphi$ is compact. \\
If $\varphi$ is compact, let $G\colon \mathcal{H}_B\rightarrow \mathcal{H}_B\in \Lambda$ be a map defined by a map $\tilde{G}\colon  B\rightarrow \mathcal{K}(H)\subset\mathcal{B}(H)$ (compare Example \ref{trivial bundle homo}). An easy computation shows that $\tilde{G}$ is continuous (the topology on $\mathcal{K}(H)$ is induced by the norm topology). Let $\varepsilon >0$ and $b\in B$. Since $ \varphi$ is compact there is a neighbourhood $U$ of $b$ and $G\in \Lambda$, 
such that for all $y\in U$ the equation $\Vert \varphi_y-G_y\Vert< \varepsilon/3$. Since $G$ is continuous there is a neighbourhood $V$ of $b$ such that $\Vert G_b-G_y\Vert <\varepsilon/3$ for all $y\in V$. Finally we compute for $
y\in U\cap V$ \[\Vert \varphi_b-\varphi_y\Vert \leq \Vert \varphi_b-G_b\Vert +\Vert G_b-G_y\Vert +\Vert G_y-\varphi_y\Vert <\frac{\varepsilon}{3}+\frac{\varepsilon}{3}+\frac{\varepsilon}{3}=\varepsilon.\]
\end{bsp}
\begin{defi}
Let $\mathcal{H}=((H(b))_{b\in B},\Gamma)$ be a continuous field of Hilbert spaces and $F\colon \mathcal{H}\rightarrow \mathcal{H}$ be a morphism. We say $F$ is a Fredholm family if there is a morphism $G\colon  \mathcal{H}\rightarrow \mathcal{H}$ such that $FG-1$ and $GF-1$ are compact.
\end{defi}

\begin{lem}\label{fredholm local}\cite[Lemma 2.16]{E1}
Let $ \mathcal{H}=((H(b))_{b\in B},\Gamma)$ be a continuous field of Hilbert space over a paracompact space $B$. Let $F\colon \mathcal{H}\rightarrow \mathcal{H}$ be a morphism. Assume there is an open cover $\mathcal{U}$ of $B$ such that $F|_U$ is a Fredholm family for every $U\in \mathcal{U}$. Then $F$ is Fredholm. 
\end{lem}

\begin{bsp}(compare \cite{E2})\label{fiberwise dirac Frdholm family}
Again we consider the fibre-wise Dirac operator of Example \ref{Fiberwise Dirac continuous field}. By Example \ref{fiberwise Dirac morphism} the fibre-wise Dirac give rise to a morphism of continuous Hilbert fields $(F_b)_{b\in B}$. We claim that $(F_b)_{b\in B}$ is a Fredholm family. Moreover the family $(D_b^2-1)_{b\in B}$ is a compact family.\\
To prove this, we only prove the second statement. The first statement follows using $F_b^2-1=\frac{-1}{D_b^2+1}$ from the definition. By Lemma \ref{fredholm local}, compactness and Fredholmness are local properties (in $B$). Therefore we can assume, that $M\rightarrow B$ is a trivial bundle, i.e. $M=X\times B$.\\
We first prove the claim if  $V=\mathbb{C}^k\times X\times B$ is a trivial bundle. Then the continuous fields of Banach spaces $W^1_B(M,V)=B\times W^1(X,\mathbb{C}^k\times X)$ and $L^2_B(M,V)=B\times L^2(X,\mathbb{C}^k\times X)$ are trivial (in the sense of Example \ref{trivial bundle}). We consider the inclusion $W^1_B(M,V)\rightarrow L^2_B(M,V)$ which is by Example \ref{trivial bundle homo} a morphism of (trivial) continuous fields. By Example \ref{trivial compact} and the Theorem of Rellich the inclusion is a compact morphism. Then $F^2-1$ is the composition \[L^2_B(M,V)\xrightarrow{\frac{1}{D^2+1}} W^1_B(M,V)\xrightarrow{inclusion} L^2_B(M,V)\] of an adjointable family and a compact family. By Lemma \ref{compact lem} $F^2-1$ is a compact family. \\
If $V$ is not trivial: Since $B$ and $X$ are compact, we can choose a finite, open cover $\mathcal{U}=\{U_i\times V_i|i=1\dots l, U_i\subset B, V_i\subset X\}$ such that $V|_{U_i\times V_i}$ is trivial. Let $\varphi_i$ be a partition of unity subordinate to $\mathcal{U}$. We now consider the operators $G_{b,i}\vcentcolon=\varphi_i(b) \cdot (\frac{1}{D_b^2+1})$ which have support in $L^2(U_i\times V_i,V|_{U_i\times V_i})$. Since $V$ is trivial over $U_i\times V_i$, we can use the first part to see that $(G_{b,i})_{b\in B}$ is a compact family for every $i$. The claim follows since $\frac{1}{D_b^2-1}=\sum_{i=1}^l G_{b,i}.$
\end{bsp}

\subsection{Continuous Fields and $K$-Theory}\label{topologie}
Using the methods of Dixmier and Douady we will finally recall a model for $K^1$ such that the fibre-wise Dirac operator $((D_X)_b)_{b\in B}$ represents an element in this model (compare \cite[Appendix A]{E1}). Using Theorem \ref{trivial theorem} we identify the spaces $L^2(M_b,V_b)$. After this identification we will see, that the map \begin{equation}\label{55} B\rightarrow Fred^1, \hspace{1cm} b\mapsto F_b\vcentcolon=\frac{(D_X)_b}{\sqrt{1+((D_X)_b)^2}}\end{equation} is not automatically continuous using the norm topology on the right hand side. We will use a weaker topology such that $b\rightarrow F_b$ is continuous. This topology is strong enough to represent $K$-theory for compactly generated, para-compact spaces. We follow in this section \cite{E1}. 
\begin{defi}
Let $B$ be a compact space. 
\begin{itemize}
\item We define the set of cycles $\{ (\mathcal{H}, \varphi)\}$ where $\mathcal{H}\vcentcolon=((H(b))_{b\in B},\Gamma)$ is a continuous fields of Hilbert spaces (with separable infinite-dimensional fibres) over $B$ and $\varphi\colon \mathcal{H}\rightarrow \mathcal{H}$ is a self-adjoint homomorphism of continuous fields such that $\varphi^2-1$ is compact (hence by definition $ \varphi$ is Fredholm).
\item The set of cycles is, using direct sum, an abelian monoid.
\item Two cycles $(\mathcal{H}_1, \varphi_1)$ and $(\mathcal{H}_2, \varphi_2)$ are called homotopic if there is a cycle $(\mathcal{E}, \psi)$ on $B\times [0,1]$ such that there are isomorphisms $a_i\colon \mathcal{E}|_{B\times\{i\}}\rightarrow  \mathcal{H}_i$ so that \[\varphi_i=a_i\circ \psi \circ (a_i)^{-1}.\]
\item We call a cycle $(\mathcal{H}, \varphi)$ acyclic if $\varphi$ is an isomorphism.
\item Let $\mathfrak{F}(B)$ be the set of homotopy classes of cycles over $B$ modulo the submonoid of homotopy classes that contain acyclic representatives.
\end{itemize}
\end{defi} 
\begin{bem}
Using direct sum $\mathfrak{F}(B)$ is a abelian group. The inverse element of  $((H(b)),\Gamma), \varphi)$ is given by $((H(b)),\Gamma), -\varphi)$ (compare \cite[Lemma 2.21]{E1}).
\end{bem}
\begin{bsp}\label{fiberwise Dirac fin}
We consider the fibre-wise Dirac operator described in Example \ref{Fiberwise Dirac continuous field}. The pair $(L_B^2(M,V),(F_b)_{b\in B})$ represents by Examples \ref{fiberwise Dirac morphism}, \ref{fiberwise Dirac selfadjoint} and \ref{fiberwise dirac Frdholm family} an element in $\mathfrak{F}(B)$.

\end{bsp}
Next we describe a topology on $Fred^1$ such that the map \eqref{55} is continuous. This topology was used for example in \cite{Bu} and \cite{E1}.
\begin{defi}\label{weak}
We define $\widetilde{Fred}$: As a set $\widetilde{Fred}=Fred^1$. We give $\widetilde{Fred}$ the weakest topology such that the evaluation and the map $\widetilde{Fred}\rightarrow \mathcal{K}(H)$ given by $F\mapsto F^2-1$ are continuous. Here $\mathcal{K}(H)$ is the set of compact operators carrying the norm topology.
\end{defi}

For a compact manifold $B$ and an infinite dimensional separable Hilbert space we define a map
\begin{equation}
\label{0} \alpha\colon [B,\widetilde{Fred}]\rightarrow \mathfrak{F}(B)
\end{equation}
as follows:
We send a map $A\colon  B\rightarrow\widetilde{Fred}$ to the cycle $(\mathcal{H}_B=B\times H, \varphi\colon (b,x)\mapsto (b,A_b(x)))$. The map $\varphi$ is a Fredholm family by Example \ref{trivial bundle homo} and Example \ref{trivial compact}. \\
We will construct a map in the other direction,

\begin{equation}\label{1}\beta\colon \mathfrak{F}(B)\rightarrow [B,\widetilde{Fred}]\end{equation} as follows: Let $J\colon H\rightarrow H$ be a self-adjoint involution and let $(\mathcal{H}=(H(b))_{b\in B},\Gamma), \varphi)\in \mathfrak{F}(B)$. We use Theorem \ref{trivial theorem} to get an isomorphism between the trivial field $\mathcal{H}_B=B\times H$ and the direct sum $\mathcal{H}_B\oplus \mathcal{H}$. Define the morphism of trivial fields \[\mathcal{H}_B\cong \mathcal{H}_B\oplus \mathcal{H}\xrightarrow{(J\oplus \varphi)}\mathcal{H}_B\oplus \mathcal{H}\cong \mathcal{H}_B.\]
By definition of a cycle the maps $(J\oplus \varphi)_b\colon H\rightarrow H$ is contained in $\widetilde{Fred}$. The map $b\mapsto (J\oplus \varphi)_b$ is continuous by Example \ref{trivial bundle homo} and Example \ref{trivial compact}. We now see that the topology of $\widetilde{Fred}$ is a natural choice coming from continuous fields of Hilbert spaces. By \cite[Appendix A]{E1} we get the following theorem:

\begin{thm}
The maps $\alpha$ and $\beta$ are well-defined mutually inverse maps.
\end{thm}

Finally to connect this to $K$-theory we consider \[id\colon  Fred^1\rightarrow \widetilde{Fred}\] which is a continuous map. By \cite[Theorem 2.22, Theorem 2.23]{E2} or \cite[Theorem 3.1.7]{BER17} or rather \cite{BJS} we get the next theorem:
\begin{thm}\label{k}
The inclusion $Fred^1\rightarrow \widetilde{Fred}$ is a weak homotopy equivalence. Therefore the inclusion induces a bijection $[B,Fred^1]\rightarrow [B,\widetilde{Fred}^1]$ if $B$ is a compact $CW$ complex.
\end{thm}
\subsection{Proof of Theorem \ref{obstruction_weak topology}}\label{proof}
In the last section we saw that a fibre-wise Dirac operator over a compact manifold $B$ represents an element in $K^1(B)$. We associate to the fibre-wise Dirac operator a continuous map $F\colon B\rightarrow \widetilde{Fred}$ for an arbitrary infinite dimensional separable Hilbert space $H$. To prove Theorem \ref{obstruction_weak topology} we use the same strategy as in the proof of Theorem \ref{obstruction_strong topology}. All we have to do is to check the compatibility of the arguments with the weaker topology. For this purpose we generalize Lemma \ref{homotopie_stark} and check that the set of invertible operators is open in $\widetilde{Fred}$:

\begin{lem}\label{invertierbar}
Let $\widetilde{Fred}$ be the space explained in \eqref{weak}. Let $\tilde{G}(H)\subset \widetilde{Fred}$ be the subset of invertible operators. Then $\tilde{G}(H)$ is open in $\widetilde{Fred}$.
\end{lem}

\begin{proof}
Let $A\in \tilde{G}(H)$. By definition the map $\widetilde{Fred}\rightarrow \mathcal{B}(H)$ given by $F\mapsto F^2-1$ is continuous where $\mathcal{B}(H)$ are the bounded operators with norm topology. We consider the map $\phi\colon  \widetilde{Fred}\rightarrow \mathcal{B}(H)$ given by $F\mapsto F^2$. This is a continuous map. Since $A$ is invertible, $A^2$ is invertible as well. The invertible operators in $\mathcal{B}(H)$ (with operator norm topology) form an open subset (this follows using the Neumann series). Therefore there is an open neighbourhood $U\subset \mathcal{B}(H)$ of $A^2$ so that every $F\in U$ is invertible. Since $\phi$ is continuous, the subset $\phi^{-1}(U)$ is open in $\widetilde{Fred}$. If $B\in \phi^{-1}(U)$, the square $B^2$ is invertible, e.g. bijective. Therefore $B$ is bijective and hence invertible. The claim follows.
\end{proof}
\begin{lem}\label{homotopie}
Let $f\colon \mathbb{R}\rightarrow \mathbb{R}$ be a continuous function such that there is a constant $0<c\leq 1$ with $f|_{\mathbb{R}\setminus (-c,c)}=id$. Let $B$ be a topological space, $A\colon B\rightarrow \widetilde{Fred}$ be a continuous map. Then the map $f(A)\colon B\rightarrow \widetilde{Fred}$, given by $f(A)_b=f(A_b)$, is a well defined continuous map. Moreover if there is a homotopy $h$ between $f$ and identity with $h|_{(-c,c)\times [0,1]}=id$, then $f(A)$ and $A$ are homotopic.
\end{lem}
\begin{proof}
For the well-definedness of $f(A)$: we have to show that $(f(A)_b)^2-1$ is a compact operator. This holds because $f|_{\mathbb{R}\setminus (-c,c)}=id$ and $A_b^2-1$ is compact.\\
Now we show that $f(A)$ is continuous. Recall that $\widetilde{Fred}$ carries the weakest topology such that the evaluation and the map $F\mapsto F^2-1$ is continuous. We consider $A$ as morphism of the trivial continuous fields (compare Example \ref{trivial bundle homo}). By Theorem \ref{functional calculus} the map $f(A)$ is again a morphism of continuous fields. Therefore the evaluation is continuous. For the second part we have to show that the map $b\mapsto (f(A)_b)^2-1$ is continuous. We know that $b\mapsto A_b^2-1$ is continuous and define the continuous map $\tilde{f}\colon \mathbb{R}_{\geq 0}\rightarrow \mathbb{R}$ given by $\tilde{f}(x)=(f(\sqrt{x}))^2$. We see that $(f(A)_b)^2=\tilde{f}(A^2_b)$. Since $\tilde{f}$ and $b\mapsto A_b^2-1$ are continuous and $\bigcup_{b\in B} spec(A_b^2)\subset \mathbb{R}_{\geq 0}$, the map $b\rightarrow \tilde{f}(A_b^2)-1$ is continuous. \\
For the last part: Let $h\colon \mathbb{R}\times [0,1]\rightarrow \mathbb{R}$ be the homotopy. We consider the map $H\colon B\times[0,1]\rightarrow \widetilde{Fred}$ given by $H(b,t)=h(A_b,t)$. This is a homotopy between $f(A)$ and $A$. 
\end{proof}

\begin{proof}[Proof of Theorem \ref{obstruction_weak topology}]
The proof is similar to the proof of Theorem \ref{obstruction_strong topology} using the Lemma \ref{invertierbar} and replacing Lemma \ref{homotopie_stark} by Lemma \ref{homotopie}.
\end{proof}

\subsection{Proof of Lemma \ref{twisted Dirac circle}}
This example was already considered in \cite{Bu} and \cite{DK} in a little different context. 

Let $B=U(k)$ be the unitary group. Recall that there is a bijection between $n$-dimensional, complex vector bundles over $S^1\wedge B$ and homotopy classes of maps $B\rightarrow U(n)$. We consider the $k$-dimensional bundle $V$ over $S^1\wedge B$ represented by the map $id\colon  U(k)\rightarrow U(k)$. Let $p\colon S^1\times B\rightarrow S^1\wedge B$ be the projection and consider $W\vcentcolon=p^\ast V$. We can construct $V$ explicitly: Let $\tilde{V}$ be the trivial bundle over $[0,1]\times U(k)$ given by \[\tilde{V}=[0,1]\times \mathbb{C}^k \times U(k).\]
Now $W\rightarrow S^1\times U(k)$ is $\tilde{V}/\sim$ where $(0,v,g)\sim (1,gv,g)$. This bundle is trivial over $\{0\}\times U(k)$, hence we consider it as bundle $V$ over $S^1\wedge U(k)$. Let $c_i^{odd}=c^{odd}_i(V)$. The cohomology ring $H^\ast(U(k))$ is the exterior algebra $\Lambda [c_1^{odd},\dots ,c_k^{odd}]$ (see for example \cite[pp. 243]{Ci}). The relation
\[c^{odd}_1(V)\smile \dots \smile c^{odd}_k(V)\not =0\in H^{\ast}(U(k))\] follows.

Now we consider the trivial bundle $q\colon  S^1\times U(k)\rightarrow U(k)$ and the vertical Clifford bundle $\Sigma(S^1)\otimes W\rightarrow S^1\times U(k)$. The fibre-wise Dirac operator represents an element in $[u]\in K^1(U(k))$. By an odd version of the Atiyah-Singer-Index theorem for families (see for example \cite[Chapter 3]{APS76}) we get \[ch^{odd}([u])=q_!(ch(W)\hat{A}(S^1))=q_!(p^\ast(ch(V))\hat{A}(S^1))\] where $q_!$ is the Gysin map associated to $q$ (for a definition see \cite{LM89}). The $\hat{A}$-genus of the circle is given by $\hat{A}(S^1)=1$. Since the Gysin map of the trivial bundle is just given by the Künneth formula $H^\ast(S^n\times X)\cong H^\ast(S^n)\otimes H^\ast(X)$ (see \cite[Chapter 4]{H02}), the composition $q_!\circ p^\ast$ is exactly the suspension isomorphism $H^\ast(S^1\wedge U(k))\xrightarrow{\cong} H^{\ast-1}(U(k))$. We see that $ch^{odd}(u)=ch^{odd}(V)$. The equality $c_i^{odd}(u)=c_i^{odd}(V)$ follows with Lemma \ref{computation odd chern}. The claim follows by Theorem \ref{obstruction_weak topology}.

\end{document}